\documentclass[a4paper, oneside]{amsart}

\usepackage{lmodern}
\usepackage[T1]{fontenc}
\usepackage[utf8]{inputenc}
\usepackage[english]{babel}

\usepackage{amsmath, amsthm, amssymb, amsaddr}
\usepackage{enumerate}
\usepackage{mathtools}
\usepackage{hyperref}

\usepackage{soul}
\usepackage[dvipsnames]{xcolor}
\theoremstyle{plain}
\newtheorem{theorem}{Theorem}[section]

\newtheorem{corollary}[theorem]{Corollary}
\newtheorem{lemma}[theorem]{Lemma}
\newtheorem{proposition}[theorem]{Proposition}
\newtheorem*{claim}{Claim}

\theoremstyle{definition}
\newtheorem{definition}[theorem]{Definition}

\theoremstyle{remark}

\newcommand{\U}{\mathcal{U}}
\newcommand{\F}{\mathcal{F}}
\newcommand{\G}{\mathcal{G}}
\newcommand{\K}{\mathcal{K}}
\newcommand{\R}{\mathbb{R}}
\newcommand{\w}{\omega}
\newcommand{\C}{\mathbb{C}}
\newcommand{\ca}{c_a}
\newcommand{\cl}{\overline}

\DeclareMathOperator{\intr}{Int}

\DeclareMathOperator{\clco}{\cl{co}}
\DeclareMathOperator{\diam}{diam}
\DeclareMathOperator{\cf}{cf}

\title{Local and global properties of spaces of minimal USCO maps}
\author{Serhii Bardyla$^1$, Branislav Novotný$^2$, and Jaroslav Šupina$^3$}

\address{$^1$University of Vienna, 
Institute of Mathematics, 
Vienna, Austria.}
\email{sbardyla@gmail.com}
\address{$^2$Slovak Academy of Sciences, Mathematical Institute, Bratislava, Slovakia}
\email{novotny@mat.savba.sk}
\address{$^3$P.J. Šafárik University in Košice, Institute of Mathematics, Košice, Slovakia}
\email{jaroslav.supina@upjs.sk}

\thanks{The research of the first named author was funded in whole by the Austrian Science Fund FWF [10.55776/ESP399]. The second and third authors were funded by APVV-20-0045. The second author was also funded by VEGA 2/0048/21 and the third one by VEGA 1/0657/22.}

\begin{document}

\begin{abstract}
In this paper, we study an interplay between local and global properties of spaces of minimal usco maps equipped with the topology of uniform convergence on compact sets. In particular, for each locally compact space $X$ and metric space $Y$, we characterize the space of minimal usco maps from $X$ to $Y$, satisfying one of the following properties: (i) compact, (ii) locally compact,  (iii) $\sigma$-compact, (iv) locally $\sigma$-compact, (v) metrizable, (vi) ccc, (vii)~locally ccc, where in the last two items we additionally assumed that $Y$ is separable and non-discrete.
Some of the aforementioned results complement ones of Ľubica Holá and Dušan Holý. Also, we obtain analogical characterizations for spaces of minimal cusco maps.
\end{abstract}
\maketitle

\section{Introduction}
Minimal usco and minimal cusco maps have been intensively studied in the literature, see the monograph~\cite{Hola2021a} and references therein. Applications of minimal usco/cusco maps emerged in the analysis of differentiability of Lipschitz functions \cite{Borwein1991, Borwein1997, Moors1995, Zajicek2012}, optimization \cite{Borwein2004, Coban1989, Kenderov1993},  selection theorems \cite{Hola2009, Hola2014, Moors2002}, and investigations of weak Asplund spaces \cite{Fabian1997, Kalenda1999, Moors2006, Preiss1990, Stegall1983}. Topological properties of spaces of minimal usco/cusco maps were investigated in~\cite{Hola2009, Hola2013b, Hola2014, Hola2015b, Hola2016, Hola2021, Hola2023, Hola2008, Hola2007b, Hola2020, Hola2022a}.  
% and \cite{Hola2019, Hola2020a}.

In this paper, we continue these investigations considering, among others, local and global compact-like properties of spaces of minimal usco/cusco maps equipped with the topology of uniform convergence on compact sets. In particular, for each locally compact space $X$ and metric space $Y$, we characterize the space of minimal usco maps from $X$ to $Y$, satisfying one of the following properties: (i) compact, (ii) locally compact,  (iii) $\sigma$-compact, (iv)~locally $\sigma$-compact, (v) metrizable, (vi) ccc, (vii)~locally ccc, where in the last two items we additionally assumed that $Y$ is separable and non-discrete.
Some of the aforementioned results complement ones of Ľubica Holá and Dušan Holý. Also, we obtain analogical characterizations for spaces of minimal cusco maps.

The paper is organized as follows. In Section~\ref{prelim}, we define necessary notions and state some auxiliary facts. Section~\ref{usco} is devoted to spaces of minimal usco maps. In Section~\ref{cusco}, we discuss spaces of minimal cusco maps.

\section{Preliminaries}\label{prelim}

A {\em space} means a nonempty Hausdorff topological space. In what follows, $X, Y, Z$ are spaces. %Let $\R,\C$ denote the space of real and complex numbers, respectively. 
The symbols $\cl A$ and $\intr (A)$ stand for the closure and the interior of the set $A$. Let $\partial A=\overline{A}\setminus A$ be the boundary of $A$. 
For a space $Y$, the family of all nonempty compact subsets of $Y$ is denoted by $K(Y)$.

A map, or a multifunction, from $X$ to $Y$ is a function that assigns to each element of $X$ a subset of $Y$. %The set of points of $X$, where a map $F$ is single-valued ($|F(x)|=1$) will be denoted $S(F)$.
If $F$ is a map from $X$ to $Y$, shortly $F\colon X\to Y$, then its graph is the set  $\{(x,y) \in X \times Y: y \in F(x)\}$. In this paper, we identify both maps and functions with their graphs. %If $f\colon X \to Y$  is a single-valued function, we will use the symbol $f$ also for the graph of $f$. 
%Notice that we can consider it to be a special case of a set-valued map. 
By $\cl F$ we denote the closure of the graph of $F$ in the Tychonoff product $X\times Y$. 
%In general, it is a set-valued map.

%A map $F$ from $X$ to $Y$ is called {\em upper semicontinuous at a point} $x \in X$ if for every open set $V$ containing $F(x)$, there exists an open set $U$ such that $x \in U$ and $F(U) \subseteq V$. 
%A map$F$ is called {\em upper semicontinuous} if it is upper semicontinuous at each point of $X$. 
A map $F$ is called {\em upper semicontinuous} if for every 
$x \in X$ and open set $V$ containing $F(x)$, there exists an open set $U$ such that $x \in U$ and $F(U) \subseteq V$.
Following Christensen \cite{Christensen1982}, we say that a map $F\colon X\to Y$ is {\em usco} if it is upper semicontinuous and $F(x)$ is a nonempty compact subset of $Y$ for each $x\in X$. Finally, a map $F\colon X\to Y$ is said to be {\em minimal usco} if it is minimal with respect to the inclusion among all usco maps from $X$ to $Y$. That is, $F$ is a minimal usco map if $F$ does not contain properly any other usco map from $X$ to $Y$.
Using the Kuratowski-Zorn principle, we can guarantee that every usco map contains a minimal usco map, see \cite{Hola2021a}.
%\cite{Borwein1999,Borwein2004,Drewnowski1990}
By $MU(X,Y)$ we denote the family of all minimal usco maps from $X$ to $Y$. Note that $MU(X,Y)\subseteq K(Y)^X$.

A convenient tool to determine whether a map is minimal usco uses notions of quasicontinuity and subcontinuity of its selections. %, see \cite{Hola2023}.
A function $f\colon X \to Y$ is {\em quasicontinuous}, if for every $x \in X$ and neighborhoods $U$ of $x$ and $V$ of $f(x)$, there is a nonempty open set $G \subseteq U$ such that $f(G) \subseteq V$, see \cite{Neubrunn1988}. %If $f$ is quasicontinuous at every point of $X$, we say that $f$ is {\em quasicontinuous}.
Let $A$ be a dense subset of $X$, then a function $f\colon A \to Y$ is said to be {\em densely defined} on $X$. Moreover, $f$ is called {\em subcontinuous} on $X$, if for every net $(x_i) \subseteq A$ convergent to some $x\in X$, the net $(f(x_i))$ has a cluster point in $Y$, see \cite{Lechicki1990}.
%A function  $f\colon A \to Y$ is {\em subcontinuous} on $X$ if it is subcontinuous at every $x \in X$.
We say that a function $f\colon A\to Y$ is a {\em dense selection} of a map $F\colon X\to Y$, if $A$ is dense in $X$ and for every $x\in A$, $f(x)\in F(X)$.
The following theorem is one of several characterizations of minimal usco maps from \cite{Hola2023}.

\begin{theorem}\label{thm:minimalusco}
	Let $Y$ be regular, and $F\colon X\to Y$. Then $F\in MU(X,Y)$, if and only if $F$ possesses a dense selection $f\colon A\to Y$ that is subcontinuous on $X$, quasicontinuous on $A$ and $\cl{f}=F$.
\end{theorem}

We will use the following notation of special maps throughout this paper.
\begin{definition}
For an open $U\subseteq X$ and $a,b\in Y$ define $f^{a,b}_U\colon X\setminus \partial U\to Y$ by
\[
f^{a,b}_U(x)=\begin{cases}
            a, & \mbox{if } x\in U, \\
            b, & \mbox{if } x\not\in\cl{U},
          \end{cases}
\]
Note that the function $f^{a,b}_U$ is densely defined on $X$, it is continuous on $X\setminus \partial U$ and subcontinuous on $X$, thus by Theorem \ref{thm:minimalusco} the map $F^{a,b}_U=\cl{f^{a,b}_U}\in MU(X,Y)$.
\end{definition}

By $H_d$ we denote Hausdorff metric on $K(Y)$ (see~\cite{Beer1993} for more details). Next, following~\cite{Hammer1997}, we shall define the topology $\tau_{uc}$ of {\em uniform convergence on compact sets} on $K(Y)^X$.

\begin{definition}\label{def:tauuc}
The topology $\tau_{uc}$ on the set $K(Y)^X$ is induced by the uniformity $\U_{uc}$, which has the following base
    \[\{W[K,\varepsilon] : K\text{ is compact and }\varepsilon>0\},\]
    where
    \[W[K,\varepsilon]=\{(F ,G) \in K(Y)^X\times K(Y)^X: H_d(F (x),G(x))<\varepsilon \text{, for every } x\in K\}.\]
    The general $\tau_{uc}$-basic neighborhood of $F\in K(Y)^X$ will be denoted by $W(F,K,\varepsilon)$, i.e., $W(F,K,\varepsilon)=W[K,\varepsilon](F)$.
\end{definition}

In this paper, the set $MU(X,Y)$ will be, by default, equipped with the subspace topology inherited from $(K(Y)^X,\tau_{uc})$.

\begin{definition}
Let $X$ be a topological space and $(Y,d)$ be a metric space. A family $\F\subseteq K(Y)^X$ is called 
\begin{itemize}
    \item[(i)] {\em densely equicontinuous at} $x\in X$, if for every $\varepsilon>0$ there is a finite family $\U$ of open subsets of $X$ such that $x\in\intr\cl{\bigcup \U}$ and for every $F\in\F$ and $U\in\U$, $\diam F(U) <\varepsilon$;
    \item[(ii)] {\em densely equicontinuous}, if $\F$ is densely equicontinuous at every $x\in X$;
    \item[(iii)] {\em pointwise bounded}, if for every $x\in X$,  the set $\bigcup\{F(x):F\in\F\}$ is relatively compact, i.e., its closure in $Y$ is compact.
\end{itemize}
\end{definition}

The following analogue of Arzela–Ascoli theorem for minimal usco maps (see \cite[Theorem 3.1]{Hola2016}) is crucial for this paper.

\begin{theorem}[Hol\'a--Hol\'y]\label{thm:arzel}
  Let $X$ be a locally compact topological space and $(Y, d)$ be a metric space. A subset $\F\subseteq MU(X,Y)$ is compact if and only if $\F$ is closed, densely equicontinuous, and pointwise bounded.
\end{theorem}

We shall use the following trivial fact a few times in sections~\ref{usco} and~\ref{cusco}.

\begin{lemma}\label{lc}
Let $X$ be a locally compact space and $K$ a compact subset of $X$. Then there exists an open set $U\subseteq X$ such that $K\subseteq U$ and $\overline{U}$ is compact.   
\end{lemma}

%As we mentioned in the introduction, Theorem \ref{thm:arzel} plays a key role in our results. 
Further {\em $X$ will always be a locally compact Hausdorff space, and $Y$ a metric space equipped with a metric $d$}. Note that under these assumptions, by \cite[Lemma~3.3]{Hola2016}, $MU(X,Y)$ is a closed subset of $(K(Y)^X,\tau_{uc})$.

For every $a\in Y$, denote a constant map with value $\{a\}$ as $\ca$. The following lemmas are trivial, so we omit their proofs.

\begin{lemma}\label{trivial1}
The function $\phi\colon Y\rightarrow MU(X,Y)$ defined by $\phi(a)=\ca$ is an isometry. Moreover, $\phi(Y)$ is closed in $MU(X,Y)$.  
\end{lemma}

\begin{lemma}\label{trivial}
If $X$ is discrete, then $MU(X,Y)=Y^X$.
\end{lemma}

%One can check that identifying $a$ with $c_a$ we embed $Y$ into $MU(X,Y)$ as a closed subset. Also, if $X$ is discrete, then it is easy to check that $MU(X,Y)$ coincides with the Tychonoff product $Y^X$.

\section{Minimal usco maps}\label{usco}

Recall that we, by default, assume that $X$ is a locally compact space and $(Y,d)$ a metric space. We start with the following simple facts necessary for the proofs of lemmas~\ref{lem:local} and~\ref{lem:global}, which are vital for this paper.

\begin{lemma}\label{lem:useful}
Let $X$ be such that $X=K\cup D$ for some compact set $K$ and discrete set $D$. Then there exists a compact open set $K'\supseteq K$. In particular, $X$ is a topological sum of a compact space $K'$ and a discrete space $X\setminus K'$.
\end{lemma}

\begin{proof}
Since $X$ is locally compact, by Lemma~\ref{lc}, there exists an open set $U\supseteq K$ such that $\cl{U}$ is compact. Fix any $x\in D\setminus U$. Since $D$ is discrete and $K$ is closed, there exists an open neighborhood $V$ of $x$ such that $V\cap K=\emptyset$ and $V\cap D=\{x\}$. Thus the set $V\cap (X\setminus K)=\{x\}$ is open in $X$, witnessing that $x\notin \cl{U}$. Since the point $x$ was chosen arbitrarily, we get that $U=\cl{U}$. Hence $X$ is a topological sum of a compact open set $U$ and a discrete set $X\setminus U$, as required.
\end{proof}

\begin{lemma}\label{lem:densely}
A family $\F\subseteq K(Y)^X$ is densely equicontinuous at $x\in X$, if and only if for every $\varepsilon>0$ and every open neighborhood $O$ of $x$ there is a finite family $\U$ of open subsets of $X$ such that 
\begin{enumerate}[{\rm(i)}]
    \item $x\in\intr\cl{\bigcup \U}$;
    \item $x\in\cl U$ and $U\subseteq O$ for every $U\in\U$;
    \item $\diam F(U) <\varepsilon$ for every $F\in\F$ and $U\in\U$.
\end{enumerate}

\end{lemma}
\begin{proof}
 The implication $(\Leftarrow)$ is trivial. 
 
 $(\Rightarrow)$: Fix an arbitrary $\varepsilon >0$. There is a finite family $\U=\{U_1,\dots,U_n\}$ of open subsets of $X$ such that $x\in\intr\cl{\bigcup \U}$ and for every $F\in\F$ and $U\in\U$, $\diam F(U) <\varepsilon$. Thus, there is an open neighborhood $V$ of $x$ such that $V\subseteq \cl{\bigcup \U}$, and also there is $U\in\U$ such that $x\in\cl U$. Let $A=\{i\leq n: x\notin\cl{U_i}\}$ and $B=\{1,\dots,n\}\setminus A$. Then for every open neighborhood $O$ of $x$ we have
    \[x\in O\cap V\cap \bigcap_{i\in A}(X\setminus \cl{U_{i}})\subseteq \bigcup_{i\in B}O\cap\cl{U_i}\subseteq \bigcup_{i\in B}\cl{O\cap U_i}\]
    and thus the family $\U'=\{O\cap U_i: i\in B\}$ satisfies the required properties.
\end{proof}

We shall deal with the following special subset of $MU(X,Y)$.

\begin{definition}
Let
  \[D_2(X,Y)=\{F\in MU(X,Y): |F(X)|\le 2 \}.\]  
\end{definition}

\begin{lemma}\label{closed}
    The set $D_2(X,Y)$ is closed in $MU(X,Y)$.
\end{lemma}

\begin{proof}
Consider any $F\in MU(X,Y)$ such that $|F(X)|>2$. Fix distinct $a,b,c\in F(X)$. Then there exists a finite subset $K\subseteq X$ such that $\{a,b,c\}\subseteq F(K)$. Fix a positive real $\varepsilon$ such that $\varepsilon<\frac{1}{2}\min\{d(a,b), d(a,c), d(b,c)\}$. At this point, it is easy to see that $W(F,K,\varepsilon)\cap D_2(X,Y)=\emptyset$. Hence the set $D_2(X,Y)$ is closed. 
\end{proof}

We say that a topological space is {\em non-trivial} if it has at least two points. Observe that $MU(X,Y)$ is non-trivial if and only if $Y$ is non-trivial.
Therefore, in the rest of the paper, we assume that $Y$ is a non-trivial metric space.
The following two lemmas play a key role in this paper. 
 
\begin{lemma}\label{lem:local}
If the subspace $D_2(X,Y)$ of $MU(X,Y)$ is locally $\sigma$-compact, then $X$ is a topological sum of a compact space and a discrete space.
Moreover, if $Y$ is non-discrete, then $X$ is discrete.
\end{lemma}

\begin{proof}
    Let us first show that $X$ is a topological sum of a compact space and a discrete space.
    Fix any $a\in Y$ and note that $\ca\in D_2(X,Y)$. Then there is a compact $K\subseteq X$ and $\eta>0$ such that 
    \begin{equation}\label{eq}
    \cl{W(\ca ,K,\eta)}\cap D_2(X,Y)=\bigcup_{n\in\omega}\F_n,
    \end{equation}
    where for each $n\in \w$ the family $\F_n$ is compact.

    By Lemma~\ref{lem:useful}, it is enough to prove that $X\setminus K$ is discrete.
    Let $a,b\in Y$ be two distinct points. Fix a positive $\varepsilon < d(a,b)$ and an arbitrary $x\in X\setminus K$. By Theorem \ref{thm:arzel}, we have that the families $\F_n$, $n\in\w$ are densely equicontinuous at $x$.
    By Lemma~\ref{lem:densely}, for every $n\in\omega$ and open neighborhood $O$ of $x$ there is a finite family $\U_{n,O}$ of open subsets of $X$ such that $x\in\intr\cl{\bigcup \U_{n,O}}$; for every $U\in\U_n$, $U\subseteq O$; and for every $F\in\F_n$ and $U\in\U_{n,O}$, $\diam F(U) <\varepsilon$.
    \begin{claim}
    There is $n\in \omega$ and an open neighborhood $O$ of $x$ such that the family $\U_{n,O}$ consists of finite open sets.
    \end{claim}
    \begin{proof}[Proof of the claim:]
    Seeking a contradiction, assume that for each $n\in \omega$ and open neighborhood $O$ of $x$ the family $\U_{n,O}$ contains an infinite element. We are going to construct a sequence of points and open sets by induction.
For $n=0$ let $O_0$ be an arbitrary open neighborhood of $x$ such that $\overline{O_0}\cap X\setminus K=\emptyset$. By the assumption, there exists an infinite $U_0\in\U_{0,O_0}$. There are $y_0,z_0\in U_0\subseteq O_0$ such that $x,y_0,z_0$ are all distinct. Choose disjoint open sets $V_0,W_0\subseteq U_0$ such that $y_0\in V_0$, $z_0\in W_0$ and $x\not\in \cl{V_0}\cup \cl{W_0}$.

    Suppose that for all $j<n$ we constructed open sets $O_j,U_j,V_j,W_j$ and points $y_j\in V_j$, $z_j\in W_j$ such that $x\in O_j$ and $x\not\in \cl{V_j}\cup \cl{W_j}$. Define an open neighborhood $O_n$ of $x$ by
    \[O_n=O_{n-1}\cap (X\setminus \cl{V_{n-1}}) \cap (X\setminus \cl{W_{n-1}}).\]
     Choose an infinite set $U_n\in\U_{n,O_n}$. There are $y_n,z_n\in U_n\subseteq O_n$ such that $x,y_n,z_n$ are all distinct. Fix disjoint open sets $V_n,W_n\subseteq U_n$ such that $y_n\in V_n$, $z_n\in W_n$ and $x\not\in \cl{V_n}\cup \cl{W_n}$.

    By the construction, all of the sets in $\{V_n,W_n:n\in\omega\}$ are pairwise disjoint. Put $V=(X\setminus \overline{O_0})\cup \bigcup_{n\in\omega} V_n$. Since $K\subseteq X\setminus \overline{O_0}$, we get that $F^{a,b}_{V}\in W(c_a,K,\eta)\cap D_2(X,Y)$. Observe that for every $n\in\w$, $F^{a,b}_{V}(U_n)=\{a,b\}$ and $d(a,b)>\varepsilon$. So $F^{a,b}_{V}\notin \F_n$ for each $n\in\w$, which contradicts Equation~\ref{eq}.
    \end{proof}

  %  The claim above implies that there exists a finite open subset $U\in \mathcal U_{n,O}$ such that $x\in \overline{U}$. 
  The claim above implies that the set $\bigcup \U_{n,O}=\intr(\overline{\bigcup\U_{n,O}})$ is a finite open neighborhood of $x$, witnessing that $x$ is an isolated point. Since $x$ was chosen arbitrarily, we obtain that $X\setminus K$ is discrete.

    \medskip
    Now, assume that $Y$ is non-discrete. Let us show that $X$ is discrete.
    There is $a\in Y$ such that for every $\eta>0$ there exists $b\in Y$ with $d(a,b)<\eta$. Then for every compact $K\subseteq X$ and open $V\subseteq X$ we have $F^{a,b}_V\in W(c_a,K,\eta)\cap D_2(X,Y)$. We repeat the previous arguments with the following changes: now $x$ is an arbitrary point of $X$ (it can belong to $K$) and $O_0=X$. This way, we show that an arbitrary point of $X$ possesses a finite open neighborhood, implying that $X$ is discrete.
\end{proof}

The proof of the following lemma is similar to the one of Lemma~\ref{lem:local}.

\begin{lemma}\label{lem:global} 
  If $D_2(X,Y)$ is $\sigma$-compact, then $X$ is discrete.
\end{lemma}

\begin{proof}
    Since the space $D_2(X,Y)$ is $\sigma$-compact, there exist compact subsets $\F_n$, $n\in\w$ of $D_2(X,Y)$ such that
$D_2(X,Y)=\bigcup_{n\in\omega}\F_n$.
    
    Fix an arbitrary $x\in X$, distinct $a,b\in Y$ and a positive $\varepsilon <d(a,b)$. By Theorem~\ref{thm:arzel}, we have that the families $\F_n$, $n\in\w$ are densely equicontinuous at $x$.
    By Lemma~\ref{lem:densely}, for every $n\in\omega$ and open neighborhood $O$ of $x$ there is a finite family $\U_{n,O}$ of open subsets of $X$ such that $x\in\intr\cl{\bigcup \U_{n,O}}$; for every $U\in\U_n$, $U\subseteq O$; and for every $F\in\F_n$ and $U\in\U_{n,O}$, $\diam F(U) <\varepsilon$. 
Similarly, as in the proof of Lemma~\ref{lem:local}, it can be shown that there exists $n\in\w$ and an open neighborhood $O$ of $x$ such that $\U_{n,O}$ consists of finite open sets. It follows that the point $x$ is isolated. Hence the space $X$ is discrete.   
\end{proof}

The following theorem characterizes compactness of a space $MU(X,Y)$.

\begin{theorem}\label{comp}
    The following  assertions are equivalent:
    \begin{enumerate}[\rm (1)]
        \item $MU(X,Y)$ is compact;
        \item $D_2(X,Y)$ is compact;   
        \item $Y$ is compact and $X$ is discrete.
    \end{enumerate}
\end{theorem}

\begin{proof}
Implication (1)$\Rightarrow$ (2) follows from Lemma~\ref{closed}. 

(2) $\Rightarrow$ (3): Note that $Y$, being homeomorphic to the closed subspace of $D_2(X,Y)$ consisting of all constant maps, is compact. Discreteness of $X$ follows from Lemma~\ref{lem:global}.

The implication (3) $\Rightarrow$ (1) follows from Lemma~\ref{trivial}. 
\end{proof}

The following two technical lemmas will be useful.
\begin{lemma} \label{lem:homeomorphism}
	Let $X$ be a topological sum of $X_1$ and $X_2$, then the space $MU(X,Y)$ is homeomorphic to $MU(X_1,Y)\times MU(X_2,Y)$.
\end{lemma}
\begin{proof}
For $i\in\{1,2\}$, define $\varphi_i\colon MU(X,Y)\to MU(X_i,Y)$ by $\varphi_i(F)=F{\restriction}X_i$. 
%It is correctly defined since $X_i$ is open.
Let us define the map \[\varphi\colon MU(X,Y)\to MU(X_1,Y)\times MU(X_2,Y)\]
by 
$$\varphi(F)=(\varphi_1(F), \varphi_2(F)), \hbox{ for every }F\in MU(X,Y).$$
One can easily verify that $\varphi$ is bijective with $\varphi^{-1}((F_1,F_2))=F_1\cup F_2$.

Note that for every compact $K\subseteq X$ we have that $K=K_1\cup K_2$, where $K_i=K\cap X_i\in K(X_i)$.
We claim that for any $F\in MU(X,Y)$, the following holds:
\[\varphi_1(W(F,K,\varepsilon))=W(F{\restriction} X_1,K_1,\varepsilon).\]
The inclusion $\varphi_1(W(F,K,\varepsilon))\subseteq W(F{\restriction} X_1,K_1,\varepsilon)$ is straightforward. For the reverse inclusion, observe that any $F_1\in W(F{\restriction} X_1,K_1,\varepsilon)$ can be extended to $G\in W(F,K,\varepsilon)$, where $G=F_1\cup F{\restriction} X_2$.
Dually, one can check that
\[\varphi_2(W(F,K,\varepsilon))=W(F{\restriction} X_2,K_2,\varepsilon),\]
and thus
\[\varphi(W(F,K,\varepsilon))=W(F{\restriction} X_1,K_1,\varepsilon)\times W(F{\restriction} X_2,K_2,\varepsilon).\]
It is easy to see that $\varphi$ is a homeomorphism.
\end{proof}

\begin{lemma}\label{lem:Xcompact}
Let $Y$ be non-compact. If $MU(X,Y)$ is locally $\sigma$-compact, then $X$ is compact.
\end{lemma}
\begin{proof}
 Seeking a contradiction, assume that $X$ is not compact. By lemmas~\ref{closed} and~\ref{lem:local}, we have that $X$ is a topological sum of a compact set $K$ and an infinite discrete set $D$. %Without loss of generality, we can assume that $K$ is nonempty.
 By Lemma~\ref{lem:homeomorphism}, we have that $MU(X,Y)$ is homeomorphic to $MU(K,Y)\times MU(D,Y)$. In particular, $MU(D,Y)$ can be embedded as a closed subset in $MU(X,Y)$. By Lemma~\ref{trivial}, $MU(D,Y)= Y^D$. Since $Y$ is a non-compact metric space, it has a countably infinite closed discrete subset $B$. Since $D$ is infinite, it contains a countable subset $C$. Thus, Baire space $\omega^\omega$ is homeomorphic to a closed subset $B^C$ of $Y^D$, i.e., it can be embedded as a closed subset of $MU(X,Y)$. Recall that $\omega^\omega$ is not locally $\sigma$-compact, a contradiction.
\end{proof}

The following theorem characterizes $\sigma$-compactness of a space $MU(X,Y)$.

\begin{theorem}\label{thm:sigma}
   A space $MU(X,Y)$ is $\sigma$-compact if and only if one of the following conditions holds: 
    \begin{enumerate}[\rm (i)]
        \item $Y$ is compact, $X$ is discrete;
        \item $Y$ is $\sigma$-compact, $X$ is finite.
    \end{enumerate}
    Moreover, in both cases we have that $MU(X,Y)=Y^X$, and in case (i) the space $MU(X,Y)$ is compact.
\end{theorem}

\begin{proof}
($\Rightarrow$): Lemmas~\ref{trivial1} and~\ref{closed} imply that $Y$ and $D_2(X,Y)$ are $\sigma$-compact. By Lemma~\ref{lem:global} we have that $X$ is discrete. So $MU(X,Y)=Y^X$, by Lemma~\ref{trivial}.
If $Y$ is compact, then Theorem~\ref{comp} yields that $MU(X,Y)$ is compact, and condition~(i) holds.
If $Y$ is not compact, Lemma~\ref{lem:Xcompact} implies that $X$ is compact. Hence $X$ is finite and condition (ii) holds. %Note that by Theorem~\ref{comp}, the space $MU(X,Y)$ is non-compact.

($\Leftarrow$): Again, by Lemma~\ref{trivial} we have that $MU(X,Y)=Y^X$. It remains to recall that a product of compact spaces remains compact, and a finite product of $\sigma$-compact spaces remains $\sigma$-compact.
\end{proof}

The following proposition characterizes discreteness of a space $MU(X,Y)$.

\begin{proposition}\label{thdiscrete}
    The following assertions are equivalent:
    \begin{enumerate}[\rm (1)]
        \item $MU(X,Y)$ is discrete;
        \item $D_2(X,Y)$ is discrete;
        \item $Y$ is discrete, $X$ is compact.
    \end{enumerate}
\end{proposition}

\begin{proof}
The implication $(1)\Rightarrow (2)$ is obvious. 
    
    $(2)\Rightarrow (3)$: Since $Y$ is a  subspace of $D_2(X,Y)$, we have that $Y$ is discrete. To derive a contradiction, assume that $X$ is not compact. Bearing in mind that $X$ is locally compact, Lemma~\ref{lc} implies that for every compact subset $K$ of $X$ there exists an open set $U(K)$ such that  $K\subseteq U(K)$ and $\overline{U(K)}$ is compact.   Thus $X\setminus \cl{U(K)}\neq \emptyset$. Fix distinct $a,b\in Y$. For every $\varepsilon >0$ and compact subset $K\subseteq X$ we have that $F^{a,b}_{U(K)}\in W(\ca,K,\varepsilon)\cap D_2(X,Y)$. Hence $\ca$ is not an isolated point of $D_2(X,Y)$, which contradicts our assumption.

    $(3)\Rightarrow (1)$: Fix any $F\in MU(X,Y)$. Since $X$ is compact, so is $F(X)$ (see \cite[Theorem 6.2.11]{Beer1993}. Hence $F(X)$ is finite. It follows there exists $\varepsilon>0$ such that $d(x,y)\geq \varepsilon$ for each $x\in F(X)$ and $y\in Y\setminus\{x\}$. At this point, it is easy to see that $W(F,X,\varepsilon)=\{F\}$. Hence the space $MU(X,Y)$ is discrete.
\end{proof}

The following theorem characterizes local compactness of a space $MU(X,Y)$.

\begin{theorem}\label{locallycomp}
    A space $MU(X,Y)$ is locally compact if and only if one of the following conditions holds:
        \begin{enumerate}[\rm (i)]
            \item $Y$ is compact, $X$ is discrete;
            \item $Y$ is discrete, $X$ is compact;
            \item $Y$ is locally compact, $X$ is finite;
            \item $Y$ is finite, $X$ is a topological sum of a compact set and a discrete set.
        \end{enumerate} 
    Moreover, the space $MU(X,Y)$ is compact in case (i) and discrete in case (ii).
\end{theorem}

\begin{proof}
($\Rightarrow$): If $MU(X,Y)$ is compact, then by Theorem~\ref{comp}, condition (i) holds. If $MU(X,Y)$ is discrete, then by Proposition~\ref{thdiscrete}, we have condition (ii).

Assume that $MU(X,Y)$ is a non-discrete locally compact non-compact space. Then $D_2(X,Y)$, being a closed subspace of $MU(X,Y)$, is locally compact. Lemma~\ref{lem:local} implies that $X$ is a topological sum of a compact space $K$ and a discrete space $D$.

Suppose that $Y$ is not discrete. We are going to show that, in this case, condition (iii) holds.
Lemma~\ref{lem:local} implies that $X$ is discrete. Lemma~\ref{trivial1} yields that $Y$ is locally compact. Since $MU(X,Y)$ is not compact, Theorem~\ref{comp} implies that $Y$ is not compact. By Lemma~\ref{lem:Xcompact}, $X$ is finite. Hence condition (iii) holds.

Suppose that $Y$ is discrete. If $Y$ is infinite, then Lemma~\ref{lem:Xcompact} implies that $X$ is compact. By Proposition~\ref{thdiscrete}, the space $MU(X,Y)$ is discrete, which contradicts our assumption. Hence $Y$ is finite, and condition (iv) holds.

$(\Leftarrow)$: Theorem~\ref{comp} implies that $MU(X,Y)$ is compact, if condition (i) holds. By Proposition~\ref{thdiscrete}, $MU(X,Y)$ is discrete, providing condition (ii) holds.
Assume condition (iii) holds. Then $MU(X,Y)=Y^X$ by Lemma~\ref{trivial}. Since finite product of locally compact spaces is locally compact, we get that $MU(X,Y)$ is locally compact.
Suppose condition (iv) holds, i.e., $X$ is a topological sum of a compact set $K$ and a discrete set $D$. If $K$ or $D$ is finite, then condition (i) or (ii) holds, respectively. So, we can assume that the sets $K$ and $D$ are infinite. By Lemma~\ref{lem:homeomorphism}, the space $MU(X,Y)$ is homeomorphic to the product of the discrete space $MU(K,Y)$ and the compact space $MU(D,Y)$. It follows that $MU(X,Y)$ is locally compact.
\end{proof}

The following theorem characterizes local $\sigma$-compactness of a space $MU(X,Y)$, complementing Theorem~\ref{locallycomp} about local compactness of $MU(X,Y)$.

\begin{theorem}\label{newlocal}
    A space $MU(X,Y)$ is locally $\sigma$-compact if and only if one of the following conditions holds:
    \begin{enumerate}[\rm (i)]
        \item $MU(X,Y)$ is locally compact;
        \item $Y$ is locally $\sigma$-compact, $X$ is finite.
    \end{enumerate}
\end{theorem}

\begin{proof}
$(\Rightarrow)$: Lemma~\ref{closed} implies that $D_2(X,Y)$ is locally $\sigma$-compact. By Lemma~\ref{lem:local}, we have that $X$ is a topological sum of a compact space $K$ and a discrete space $D$. Lemma~\ref{trivial1} implies that $Y$ is locally $\sigma$-compact.

If $Y$ is neither discrete nor compact, then lemmas~\ref{closed} and~\ref{lem:local} imply that $X$ is discrete, and Lemma~\ref{lem:Xcompact} implies that $X$ is compact. Hence $X$ is finite, and condition (ii) holds.

If $Y$ is both discrete and compact, then $Y$ is finite and, by Theorem~\ref{locallycomp}(iv), $MU(X,Y)$ is locally compact.
If $Y$ is discrete but not compact, Lemma~\ref{lem:Xcompact} implies that $X$ is compact and, by Theorem~\ref{locallycomp}(ii), $MU(X,Y)$ is locally compact.
If $Y$ is compact but not discrete, lemmas~\ref{closed} and~\ref{lem:local} imply that $X$ is discrete and, by Theorem~\ref{locallycomp}(i), $MU(X,Y)$ is locally compact.

The implication $(\Leftarrow)$ follows from the fact that a finite product of locally $\sigma$-compact spaces is locally $\sigma$-compact.
\end{proof}

The following corollary of Theorem~\ref{newlocal} characterizes local $\sigma$-compactness of a space $MU(X,Y)$, for $X,Y$ being infinite. Also, it shows the equivalence between a local and a global property of a space $MU(X,Y)$.

%Note that in this case, a local property of a space $MU(X,Y)$ is equivalent to a global property.

\begin{corollary}
Let $X$ and $Y$ be infinite. $MU(X,Y)$ is locally $\sigma$-compact if and only if $MU(X,Y)$ is either compact or discrete.
\end{corollary}

\begin{proof}
Since the space $X$ is infinite, Theorem~\ref{newlocal} implies that $MU(X,Y)$ is locally compact. Taking into account that the space $Y$ is infinite, Theorem~\ref{locallycomp} implies that $MU(X,Y)$ is either compact or discrete.
The reverse implication is trivial.
%If $MU(X,Y)$ is compact, then Theorem~\ref{comp} implies that $Y$ is compact and $X$ is discrete. In this case, by Lemma~\ref{trivial}, $MU(X,Y)=Y^X$. If $MU(X,Y)$ is discrete, then Lemma~\ref{thdiscrete} yields that $Y$ is discrete and $X$ is compact. 
\end{proof}

The following result will be crucial in the next section.

\begin{proposition}\label{precusco}
  Let $Y$ be a~topological vector space over $\R$ or $\C$ equipped with a~complete metric $d$. Then the following assertions are equivalent:
  \begin{enumerate}[\rm (1)]
    \item $MU(X,Y)$ is locally compact,
    \item $MU(X,Y)$ is $\sigma$-compact,
    \item $MU(X,Y)$ is locally $\sigma$-compact,
    \item $Y$ is finite-dimensional and $X$ is finite.
  \end{enumerate}
\end{proposition}
\begin{proof}
   The implications (1) $\Rightarrow$ (3) and (2) $\Rightarrow$ (3) are trivial.
   
   To prove (3) $\Rightarrow$ (4) assume that $MU(X,Y)$ is locally $\sigma$-compact. Since $Y$ is a non-trivial topological vector space, it is neither compact nor discrete. Then by theorems~\ref{locallycomp} and~\ref{newlocal} we have that $Y$ is locally $\sigma$-compact and $X$ is finite. 
   %Since $Y$ is topological vector space, it is easy to see that $Y$ is also $\sigma$-compact. 
   Since $Y$ is completely metrizable, then $Y$ is Baire, i.e., no open subset of $Y$ is meager. The local $\sigma$-compactness of $Y$ yields an open neighborhood $V$ of $0$, which is covered by a union of countably many compact sets. Then at least one of these sets has nonempty interior, witnessing that $Y$ is locally compact. By \cite[Proposition 4.13]{Fabian2013}, a topological vector space is locally compact if and only if it is finite-dimensional. %{\color{blue}It is easy to see that in the latter case, it is also $\sigma$-compact.}
   
   The implication (4) $\Rightarrow$ (1) follows from Lemma~\ref{trivial} and \cite[Proposition 4.13]{Fabian2013}, mentioned above.
   
The implication (4) $\Rightarrow$ (2) follows from Lemma~\ref{trivial}, \cite[Proposition 4.13]{Fabian2013} and the folklore fact that every locally compact topological vector space is $\sigma$-compact.
\end{proof}

%Let $Y$ be a topological vector space. Using Theorem~\ref{comp} it is easy to observe that $MU(X,Y)$ is compact if and only if $MU(X,Y)$ is singleton.

%Note that if $Y$ is a topological vector space, then compactness of $MU(X,Y)$ implies, by Theorem~\ref{comp}, that $Y$ is compact and hence trivial. In this case, $MU(X,Y)$ itself is trivial.

A space $Z$ is called
\begin{enumerate}[\rm (i)]
    \item {\em ccc} if every family of open pairwise disjoint subsets of $Z$ is countable;
    \item {\em locally ccc}, if every $z\in Z$ has a neighborhood which is a ccc subspace of $Z$.
\end{enumerate}

Recall that a metric space is ccc if and only if it is separable.
The following result complements Lemma~4.8 and Corollary 4.10 from \cite{Hola2021}. It also illustrates yet another equivalence between a local and a global property of a space $MU(X,Y)$.

\begin{theorem}\label{thm:ccc}
Let $Y$ be non-discrete and separable. The following conditions are equivalent:
\begin{enumerate}[\rm (1)]
\item $MU(X,Y)$ is ccc;
\item $MU(X,Y)$ is locally ccc;
\item $X$ is discrete.
\end{enumerate}
\end{theorem}

\begin{proof}
The implication (1) $\Rightarrow$ (2) is obvious.

(2) $\Rightarrow$ (3).
Since $X$ is locally compact, it is sufficient to show that each compact subset of $X$ is finite.
Seeking a contradiction, assume that $X$ contains an infinite compact set $A$. It is easy to see that there exists a sequence $\{V_n:n\in\w\}$ of pairwise disjoint open sets such that $V_n\cap A\neq \emptyset$ for all $n\in\w$. %Let $A=\bigcup_{n\in\w}V_n$.
Since $Y$ is not discrete, it contains a nonisolated point $a$. Since $MU(X,Y)$ is locally ccc, there exists $\varepsilon>0$ and a compact space $B$ such that the subspace $W(\ca,B,\varepsilon)$ is ccc. Observe that for $K=A\cup B$ the subspace $W(\ca,K,\varepsilon)$ is ccc as well.  Fix a point $b\in Y$ such that $d(a,b)=\delta<\varepsilon/2$. For each function $h\colon\w\rightarrow \{a,b\}$, let
$W(h)=\bigcup_{n\in h^{-1}(a)}V_n$. Then $F_h= F_{W(h)}^{a,b}\in MU(X,Y).$ It is a routine to check that $W(F_h,K,\delta/4)\cap W(F_g,K,\delta/4)=\emptyset$ for any distinct $h,g\in \{a,b\}^\w$, and
\[\bigcup_{h\in \{a,b\}^\w}W(F_h,K,\delta/4)\subseteq W(\ca,K,\varepsilon).\]
It follows that $W(\ca,K,\varepsilon)$ is not ccc, which is a desired contradiction.
Hence each compact subset of $X$ is finite, and thus $X$ is discrete.

(3) $\Rightarrow$ (1). By Lemma~\ref{trivial}, $MU(X,Y)=Y^X$. Since $Y$ is separable, then for any finite subset $F\subseteq X$ the Tychonoff product $Y^F$, being separable and metrizable, is ccc.  By \cite[II.1.9]{Kunen2014}, the space $MU(X,Y)$ is ccc.
\end{proof}

For a partially ordered set $(P,\le)$ let
\[\cf(P,\le)=\w+\min\{|Q|:Q\text{ is cofinal in }(P,\le)\}.\]
%Recall that $K(Z)$ is the family of all compact subsets of a space $Z$.
A {\em tightness} of a space $Z$ is the minimal cardinality $t(Z)\geq \omega$ such that for every $z\in Z$ and every $A\subseteq Z$ with $z\in\cl A$ there is $A_0\subseteq A$ satisfying $|A_0|\le t(Z)$ and $z\in\cl A_0$.
A {\em character} of a space $Z$ is the minimal cardinality $\chi(Z)\geq \omega$ such that every $z\in Z$ has an open neighborhood base of cardinality at most $\chi(Z)$.
If $(Z,\U)$ is a uniform space, then by $w(Z,\U)$ we denote the {\em weight} of the uniformity $\U$, i.e., $w(Z,\U)=\cf(\U,\supseteq)$.

The next theorem extends Theorem~4.2 from \cite{Hola2021}.

%\begin{theorem}[Hol\'a, Hol\'y]
%    $\chi(MU(X,\mathbb R))=\cf(K(X),\subseteq)$.
%\end{theorem}

\begin{proposition}\label{tight}
The following equalities hold:
  \[t(MU(X,Y))=\chi(MU(X,Y))=w(MU(X,Y),\U_{uc})=\cf(K(X),\subseteq).\]
\end{proposition}
\begin{proof}
  The uniformity $\U_{uc}$ is generated by $W[K,\frac{1}{n+1}]$, where $K$ runs through a cofinal subset $\K$ of $K(X)$ and $n\in\omega$. At this point, it is easy to see that the following inequalities hold:
  \[t(MU(X,Y))\le \chi(MU(X,Y))\le w(MU(X,Y),\U_{uc}) \le \cf(K(X),\subseteq).\]

Let $\lambda=t(MU(X,Y))$.
  To prove $\cf(K(X),\subseteq)\le \lambda$, fix two distinct points $a,b\in Y$ and consider the family $\F=\{F^{a,b}_{\intr K}:K\in \K\}$. Since $X$ is locally compact, by Lemma~\ref{lc}, for every $K\in K(X)$ there is a compact $Q$ such that $K\subseteq \intr Q$. Thus $F^{a,b}_{\intr Q}\in\F\cap W(\ca ,K,\varepsilon)$, for every $\varepsilon >0$. It follows that $\ca \in\cl\F$. There exists a subfamily $\F_0=\{F^{a,b}_{\intr K_\alpha}:\alpha<\lambda\}\subseteq \F$ such that $\ca \in\cl{\F_0}$. In order to show that the family $\{K_\alpha: \alpha<\lambda\}$ is cofinal in $K(X)$, fix an arbitrary $K\in K(X)$ and positive $\varepsilon <d(a,b)$. Since $\ca \in\cl{\F_0}$, there exists $\alpha<\lambda$  such that $F^{a,b}_{\intr K_\alpha}\in W(\ca,K,\varepsilon)$. The choice of $\varepsilon$ implies that $K\subseteq K_\alpha$. Thus the family $\{K_\alpha:\alpha<\lambda\}$ is cofinal in $K(X)$.
\end{proof}

%A topological space $Z$ is called {\em locally metrizable}, if for every $z\in Z$ there exists an open neighborhood $U$ of $z$ such that the subspace $U$ is metrizable. 
A space $Z$ is called 
\begin{enumerate}[\rm (i)]
\item {\em Fréchet-Urysohn} if for every $E\subseteq Z$ and $z\in\overline{E}$ there is a~sequence $(z_n)$ in~$E$ converging to~$z$;
\item {\em sequential}, if for every non-closed set $A\subset Z$ there exists a convergent sequence $(z_n)\subseteq A$ whose limit is not in $A$;
\item {\em hemicompact} if $\cf(K(Z),\subseteq)=\w$.
\end{enumerate}

The following theorem generalizes Corollary~4.5 from \cite{Hola2021}, where the equivalence of items (1), (2), and (6) was shown for $Y=\mathbb R$. 

\begin{theorem}\label{cor:metrizable}
The following conditions are equivalent:
    \begin{enumerate}[\rm (1)]
      \item $MU(X,Y)$ is metrizable,
      %\item $MU(X,Y)$ is locally metrizable,
      \item $MU(X,Y)$ is first countable,
      \item $MU(X,Y)$ is Fréchet-Urysohn,
      \item $MU(X,Y)$ is sequential,
      \item $MU(X,Y)$ has countable tightness,
      \item $X$ is hemicompact.
    \end{enumerate}
\end{theorem}

\begin{proof}
The chain of implications $(1) \Rightarrow (2) \Rightarrow (3) \Rightarrow (4) \Rightarrow (5)$ is trivial. The implication $(5) \Rightarrow (6)$ follows from Proposition~~\ref{tight}. One can easily see that if $X$ is hemicompact, then $\U_{uc}$ has a countable base, so the space $MU(X,Y)$ is metrizable. 
%Compatible metric can be found in \cite{Hola2023}.
\end{proof}

A space $Z$ is called {\em locally second-countable} if every $z\in Z$ has a neighborhood which is a second-countable subspace of $Z$.

\begin{theorem}\label{thm:loc2count}
Let $Y$ be non-discrete and separable. The following conditions are equivalent:
\begin{enumerate}[\rm (1)]
\item $MU(X,Y)$ is second-countable metrizable;
\item $MU(X,Y)$ is locally second-countable;
\item $X$ is countable discrete.
\end{enumerate}
\end{theorem}

\begin{proof}
The implication (1) $\Rightarrow$ (2) is trivial.

(2) $\Rightarrow$ (3):  Assume that $MU(X,Y)$  is locally second-countable. Then $MU(X,Y)$  is first-countable and locally ccc.  By Corollary~\ref{cor:metrizable}, $X$ is hemicompact.  Theorem~\ref{thm:ccc} implies that $X$ is discrete. Hence $X$ is countable.

(3) $\Rightarrow$ (1):
The separability of $Y$ implies that $Y$ is second-countable. Since $X$ is countable discrete, $MU(X,Y)=Y^X$ is a countable product of second-countable metrizable spaces, which is also second-countable and metrizable.
\end{proof}

Note that the latter two results provide another equivalence between local and global properties of the space $MU(X,Y)$.

\section{Minimal cusco maps}\label{cusco}

A map $F$ from a space $X$ to a topological vector space $Y$ over $\R$ or $\C$ is called {\em cusco} if it is usco and $F(x)$ is convex for every $x \in X$. Moreover, $F$ is {\em minimal cusco} if it is minimal with respect to the inclusion among all cusco maps from $X$ to $Y$. While working with cusco maps it is convenient to assume that the vector space $Y$ is additionally locally convex, in particular, this assumption has been customary in \cite{Hola2021a} and \cite{Hola2023}.
So, in this section, we by default assume that $X$ is a locally compact space and {\em $Y$ is a non-trivial locally convex topological vector space over $\R$ or $\C$, equipped with a compatible metric $d$}.

Let $MC(X,Y)$ be the space of all minimal cusco maps from $X$ to $Y$ equipped with the
subspace topology inherited from $(K(Y)^X,\tau_{uc})$ (see Definition~\ref{def:tauuc}).
Since we assume that $X$ is a locally compact space and $Y$ is a topological vector space, by \cite[Theorem 7.6]{Hola2020}, both $MU(X,Y)$ and $MC(X,Y)$ can be turned into topological vector spaces by defining appropriate operations.
If $Y$ is a complete metric space, then, by \cite[Theorem 4.5]{Hola2022a}, $MU(X,Y)$ and $MC(X,Y)$ are homeomorphic (even isomorphic as topological vector spaces). Thus, we have the following consequence of Proposition~\ref{precusco}.

\begin{corollary}
  If the metric on $Y$ is complete, then the following assertions are equivalent:
  \begin{enumerate}[\rm (1)]
    \item $MC(X,Y)$ is locally compact,
    \item $MC(X,Y)$ is $\sigma$-compact,
    \item $MC(X,Y)$ is locally $\sigma$-compact,
    \item $Y$ is finite-dimensional and $X$ is finite.
  \end{enumerate}
\end{corollary}

\begin{definition}
For each $a,b\in Y$ and open subset $U$ of $X$, define a minimal cusco map $G^{a,b}_U$ by $G^{a,b}_U(x)=\clco [F^{a,b}_U(x)]:=$ the closed convex hull of $F^{a,b}_U(x)$.
Indeed, the values of $G^{a,b}_U$ are either $\{a\},\{b\}$ or a line segment from $a$ to $b$, which is always compact. Thus, bearing in mind that $Y$ is locally convex, by \cite[Theorem 2]{Hola2023}, we have $G^{a,b}_U\in MC(X,Y)$.
\end{definition}

\begin{theorem}\label{thm:cccc}
Let $Y$ be separable. The following conditions are equivalent:
\begin{enumerate}[\rm (1)]
\item $MC(X,Y)$ is ccc;
\item $MC(X,Y)$ is locally ccc;
\item $X$ is discrete.
\end{enumerate}
\end{theorem}

\begin{proof}
The implication (1) $\Rightarrow$ (2) is obvious. The proof of the implication (3) $\Rightarrow$ (1) is similar to the one in the proof of Theorem~\ref{thm:ccc}.

(2) $\Rightarrow$ (3).
Since $X$ is locally compact, it is sufficient to show that each compact subset of $X$ is finite.
Seeking a contradiction, assume that $X$ contains an infinite compact set $A$. It is easy to see that there exists a sequence $\{V_n:n\in\w\}$ of pairwise disjoint open sets such that $V_n\cap A\neq \emptyset$ for all $n\in\w$.

Since $Y$ is not discrete, it contains a nonisolated point $a$. Since $MC(X,Y)$ is locally ccc, there exists $\varepsilon>0$ and a compact space $B$ such that the subspace $W(\ca,B,\varepsilon)$ is ccc. Observe that for $K=A\cup B$ the subspace $W(\ca,K,\varepsilon)$ is ccc as well.  Fix a point $b\in Y$ such that $d(a,b)=\delta<\varepsilon/2$. For each function $h\colon\w\rightarrow \{a,b\}$, let
$W(h)=\bigcup_{n\in h^{-1}(a)}V_n$. Then the map $G_h= G_{W(h)}^{a,b}\in MC(X,Y).$ It is a routine to check that $W(G_h,K,\delta/4)\cap W(G_g,K,\delta/4)=\emptyset$ for any distinct $h,g\in \{a,b\}^\w$, and
\[\bigcup_{h\in \{a,b\}^\w}W(G_h,K,\delta/4)\subseteq W(\ca,K,\varepsilon).\]
It follows that $W(\ca,K,\varepsilon)$ is not ccc, which is a desired contradiction.
Hence each compact subset of $X$ is finite, and thus $X$ is discrete.
\end{proof}

The proof of the following theorem is analogical to the one of Proposition \ref{tight}, with a single change: one should consider the family $\G=\{G^{a,b}_{\intr K}:K\in \K\}$ instead of the family $\F=\{F^{a,b}_{\intr K}:K\in \K\}$.

\begin{proposition}
The following equalities hold: 
   \[t(MC(X,Y))=\chi(MC(X,Y))=w(MC(X,Y),\U_{uc})=\cf(K(X),\subseteq).\]
\end{proposition}

Analogically to Corollary \ref{cor:metrizable} we have:

\begin{corollary}\label{MCcor}
 The following assertions are equivalent:
    \begin{enumerate}[\rm (1)]
      \item $MC(X,Y)$ is metrizable;
      \item $MC(X,Y)$ is first countable;
      \item $MC(X,Y)$ is Fréchet-Urysohn;
      \item $MC(X,Y)$ is sequential;
      \item $MC(X,Y)$ has countable tightness;
      \item $X$ is hemicompact.
    \end{enumerate}
\end{corollary}

The proof of the following theorem is analogous to the proof of Theorem~\ref{thm:loc2count}, where instead of Theorem~\ref{thm:ccc} and Corollary~\ref{cor:metrizable}, one should use Theorem~\ref{thm:cccc} and Corollary~\ref{MCcor}, respectively. 

\begin{theorem}
Let $Y$ be separable. The following conditions are equivalent:
\begin{enumerate}[\rm (1)]
\item $MC(X,Y)$ is second-countable metrizable;
\item $MC(X,Y)$ is locally second-countable;
\item $X$ is countable discrete.
\end{enumerate}
\end{theorem}

%\textbf{Acknowledgements.}

%\bibliographystyle{abbrv}
%\bibliography{topology}

\end{document}